\documentclass[12pt,a4paper]{asl}

\usepackage{xspace}
\usepackage{dsfont}

\title{Subcompact Cardinals, Squares, and Stationary Reflection}

\author{Andrew D. Brooke-Taylor}
\revauthor{Brooke-Taylor, Andrew D.}
\author{Sy-David Friedman}
\revauthor{Friedman, Sy-David}

\address{Graduate School of System Informatics\\
Kobe University\\
Rokko-dai 1-1, Nada\\
Kobe, 657-8501, Japan}
\email{andrewbt@kurt.scitec.kobe-u.ac.jp}

\address{Kurt G\"odel Research Center for Mathematical Logic\\
W\"ahringer Strasse 25\\
A1020, Vienna, Austria}
\email{sdf@logic.univie.ac.at}

\thanks{Andrew Brooke-Taylor would like to thank the
Heilbronn Institute for Mathematical Research for their support during this research.  
Sy Friedman would like to thank the Austrian Science Fund (FWF) 
for its support through grant P20835-N13.}

\newcommand{\al}{\alpha}
\newcommand{\be}{\beta}
\newcommand{\ga}{\gamma}
\newcommand{\de}{\delta}
\newcommand{\ka}{\kappa}
\newcommand{\la}{\lambda}
\renewcommand{\phi}{\varphi}

\newcommand{\Power}{\mathcal{P}}
\newcommand{\st}{\mid}
\newcommand{\forces}{\Vdash}
\newcommand{\sat}{\vDash}

\newcommand{\C}{\mathbb{C}}

\renewcommand{\P}{\mathbb{P}}

\newcommand{\R}{\mathbb{R}}
\newcommand{\bbS}{\mathbb{S}}
\newcommand{\bbT}{\mathbb{T}}

\newcommand{\calC}{\mathcal{C}}

\newcommand{\restr}{\!\upharpoonright\!}

\DeclareMathOperator{\Cof}{Cof}
\DeclareMathOperator{\cf}{cf}
\DeclareMathOperator{\crit}{cp}
\DeclareMathOperator{\dom}{dom}

\newcommand{\Lim}{\textrm{Lim}}

\DeclareMathOperator{\ot}{ot}

\DeclareMathOperator{\SR}{SR}

\newcommand{\calU}{\mathcal{U}}

\newtheorem{thm}{Theorem}
\newtheorem{defn}[thm]{Definition}
\newtheorem{lemma}[thm]{Lemma}

\newtheorem{prop}[thm]{Proposition}

\renewcommand{\iff}{\leftrightarrow}
\newcommand{\of}{\circ}

\newcommand{\note}[1]{\relax}

\begin{document}

\begin{abstract}
We generalise Jensen's result on 
the incompatibility of subcompactness with $\square$.
We show that $\al^+$-subcompactness of 
some cardinal less than or equal to $\al$ 
precludes $\square_\al$,
but also that square may be forced to hold 
everywhere where this obstruction is not present.
The forcing also preserves other strong large cardinals.
Similar results are also given for stationary reflection, with a 
corresponding strengthening of the large cardinal assumption involved.
Finally, we refine the analysis by considering Schimmerling's hierarchy of 
weak squares, 
showing which cases are precluded by $\al^+$-subcompactness, and 
again we demonstrate the optimality of our results by
forcing the strongest possible squares under these restrictions to hold.
\end{abstract}

\maketitle

\section{Introduction}

A well known result of Solovay~\cite{Sol:SCCG}
is that $\square_\al$ must fail for all
$\al$ greater than or equal to a supercompact cardinal $\ka$.
Gregory improved the result, showing that strongly compact $\ka$ sufficed,
and Jensen refined it, showing that if $\ka$ is merely subcompact then
$\square_\ka$ fails (see for example \cite[Proposition~8]{SDF:LCL}).  
Jensen's result can be seen to be more or less optimal for 
$\square_\ka$ with $\ka$ a large cardinal, as 
Cummings and Schimmerling~\cite[Section 6]{CuS:ISq} have shown that 
one can force $\square_\ka$ to hold for $\ka$ 1-extendible, 
a property just short of subcompactness.
Moreover, Schimmerling and Zeman~\cite{SchZ:CSCM} have shown that
subcompactness is the only possible obstacle to square for $L[\vec{E}]$ models.
However, as is shown below, forcing $\square_\al$ at all cardinals which are not
subcompact necessarily entails the destruction of stronger large cardinal
properties.
Moreover, $\square_\ka$ can hold for $\ka$ a Vop\v{e}nka cardinal,
a consistency-wise stronger assumption which however does not directly imply subcompactness
--- see \cite{Me:IVP}.

In this article we obtain an optimal result regarding the consistency of
square with large cardinals.
Specifically, we show that
$\square_\al$ must fail whenever there is a $\ka\leq\al$ that is 
$\al^+$-subcompact (appropriately defined).  Also,
under the GCH, $\square_\al$
may be forced to hold at all other cardinals, preserving all
$\be$-subcompact cardinals of the ground model for all $\be$, along with other
large cardinals, of which we give $\omega$-superstrong cardinals as an example.

Stationary reflection is a principle 
which may be viewed as a strong negation of $\square$.
With this strengthening comes a strengthening of the large cardinal needed
to imply it: we show that if some $\ka$ is 
\emph{$\al^{+}$-stationary subcompact} (see Definition~\ref{statsub}), then
stationary reflection holds at $\al^+$.  
Moreover this is again in some sense optimal: under GCH, we can force to 
have stationary reflection fail everywhere where 
stationary subcompactness does not require
it to hold (and indeed, having $\square_\al$ everywhere possible, as above),
whilst preserving the pattern of restrictions due to subcompactness and
stationary subcompactness, as well as $\omega$-superstrong cardinals.

There are many possible
weakenings of square that have be considered in the literature.
Apter and Cummings~\cite{ApC08:LMVLC} have even formulated a form of
square compatible with supercompactness, and furthermore compatible 
with the existence of many supercompact cardinals.
In Section~\ref{weakersquares}
we consider the 
hierarchy of weak forms of
$\square_\al$ introduced by Schimmerling~\cite{Sch:CPCM1W}.
We show that known results ruling out such weak squares from 
supercompact cardinals have subcompactness analogues.
Moreover, under the GCH these results are again optimal, 
as we are able to force to obtain a 
universe in which for every cardinal $\al$ the strongest form of $\square_\al$
not so precluded holds.  

\section{Preliminaries}\label{prelims}

For any regular carinal $\al$, we denote by $H_\al$ the set of all sets
of hereditary cardinality strictly less than $\al$.
We denote by $\Lim$ the class of limit ordinals, and by $\Cof(\al)$
the class of ordinals of cofinality $\al$.
For any set of
ordinals $C$, we denote by $\ot(C)$ the order type of $C$ and by
$\lim(C)$ the set of limit points of $C$.

\begin{defn}
For any cardinal $\al$, a \emph{$\square_\al$-sequence} is a sequence
$\langle C_\be\st\be\in\al^+\cap\Lim\rangle$ such that for every 
$\be\in\al^+\cap\Lim$,
\begin{itemize}
\item $C_\be$ is a closed unbounded subset of $\be$,
\item $\ot(C_\be)\leq\al$,
\item for any $\ga\in\lim(C_\be)$, $C_\ga=C_\be\cap\ga$.
\end{itemize}
We say $\square_\al$ holds if there exists a $\square_\al$-sequence.
\end{defn}


The principle $\square_\al$ should be viewed as a property of
$\al^+$ rather than $\al$: indeed, we shall use below the fact that
$\square_\al$ can be forced over a model of GCH without changing 
$H_{\al^+}$.  The point is also emphasized by the relationship of $\square$
to \emph{stationary reflection}.

\begin{defn}\label{SR}
For $\ka>\la$ both regular, $\SR(\ka,\la)$ is the statement that for every
stationary subset $S$ of $\ka\cap\Cof(\la)$, there is a $\ga<\ka$ such that
$S\cap\ga$ is stationary in $\ga$.
\end{defn}

Note that $\square_\al$ refutes $\SR(\al^+,\la)$ for every $\la\leq\al$: 
the function
$\xi\mapsto\ot(C_\xi)$ 
from $(\al^+\smallsetminus\al+1)\cap\Cof(\la)$ to $\al+1$ is 
regressive, and so is constant on a stationary set $S$.  But now if
$S\cap\ga$ is stationary in $\ga$, then a pair of distinct elements of
$S\cap\lim(C_\ga)$ can be found, violating coherence.

We now define the large cardinals that we shall be considering.

\begin{defn}\label{subcompact}
For any 
cardinal $\al$, we say that 
a cardinal $\ka<\al$ is \emph{$\al$-subcompact} if for every
$A\subseteq H_\al$, there exist $\bar\ka<\bar\al<\ka$,
$\bar A\subseteq H_{\bar\al}$, and an elementary embedding 
\[
\pi:(H_{\bar{\al}},\in,\bar A)\to(H_\al,\in,A)
\]
with critical point $\bar\ka$ such that $\pi(\bar\ka)=\ka$.
We say that such an embedding $\pi$
\emph{witnesses the $\al$-subcompactness of $\ka$ for $A$}.
If $\ka<\al$ and 
$\ka$ is $\be$-subcompact for every $\be$ strictly between $\ka$ and $\al$,
we say that $\ka$ is \emph{$<\al$-subcompact}.
\end{defn}

Note that $\ka^+$-subcompactness of $\ka$ 
is Jensen's original notion of subcompactness.
Also note that since a finite sequences of subsets of $H_\al$ 
may be encoded into a single subset of $H_\al$
(for example, with pairs $(i,x)$ for $x$ in the $i$th subset),
we may use structures with any finite number of sets $A_i$ 
rather than just one.
Typical arguments show that if $\ka$ is $\al$-subcompact then
$\ka$ is inaccessible, and indeed we shall show below that it is
a very much stronger large cardinal assumption,
culminating in
$\al$-subcompactness for all $\al$ being
equivalent to supercompactness.

If $\ka$ is $\al$-subcompact and $\ka<\be<\al$, then 
$\ka$ is also $\be$-subcompact.  Further, if the GCH holds 
then $H_{\al^+}$ contains all the
necessary sets to witness that $\ka$ is $\al$-subcompact.
Thus, if $\pi:(H_{\bar\al^+},\in,\bar A)\to(H_{\al^+},\in,A)$
with $\crit(\pi)=\bar\ka$
witnesses $\al^+$-sub\-com\-pact\-ness of $\ka$ 
for some (arbitrary) $A\subseteq H_{\al^+}$, then
$\bar\ka$ is $\bar\al$-subcompact by elementarity.
Further, if $\al$ is a limit cardinal and
$\pi:(H_{\bar\al},\in,\bar A)\to(H_{\al},\in,A)$ with critical point $\bar\ka$
witnessess $\al$-subcompactness of $\ka$ with respect to $A$,
then $\bar\ka$ is  $<\bar\al$-subcompact.

The requirement in Definition~\ref{subcompact} that $\bar\al$ be less than 
$\ka$ is a natural one similar to those that are made for 
a variety of other large cardinal axioms:
for example, the requirement that $j(\ka)>\la$ for $j$ an embedding
witnessing the $\la$-supercompactness of some $\ka$.
As in those cases, this restriction is mostly just a convenience, only ruling
out circumstances which are consistency-wise much stronger, as we shall 
now demonstrate.
To this end, let us define a cardinal $\ka$ to be 
\emph{loosely $\al$-subcompact} if it satisfies the requirements to be
$\al$-subcompact except that the cardinal $\bar\al$ 
need not be less than $\ka$.

First note that we may usually assume
that $\bar\al$ is strictly less than $\al$.
\begin{lemma}\label{barless}
Suppose $\ka$ is loosely $\al$-subcompact, $\cf(\al)>\omega$,
and $A\subseteq H_{\al}$.
Then there is an elementary embedding 
$\pi:(H_{\bar\al},\in,\bar A)\to\,$\mbox{$(H_{\al},\in,A)$} witnessing the
$\al$-subcompactness of $\ka$ for $A$ such that $\bar\al$ is strictly less
than $\al$.
\end{lemma}
\begin{proof}
The proof is just as in Kunen's proof \cite{Kun:ElEm} that there
can be no nontrivial elementary embedding from $V_{\la+2}$ to $V_{\la+2}$.
Specifically, let $f:[\al]^\omega\to \al$ be $\omega$-J\'onsson for $\al$,
that is, for any subset $X$ of $\al$ of cardinality $\al$, 
$f``[X]^\omega=\al$; such functions were shown to exist for all $\al$ by 
Erd\H{o}s and Hajnal \cite{ErH:PJ}.
Note that $f$ is a subset of $H_{\al}$,
so there will be an elementary embedding
\mbox{$\pi:(H_{\bar\al},\in,\bar A,\bar f)$}$\to(H_{\al},\in,A,f)$
witnessing the $\al$-subcompactness of $\ka$ for $A$ and $f$.
We claim that this $\pi$, when considered as a function from
$(H_{\bar\al},\in,\bar A)$ to $(H_{\al},\in,A)$, 
satisfies the requirements
of the lemma, namely, that $\bar\al<\al$.
For suppose $\bar\al$ were to equal $\al$.
Then since $|\pi``\bar\al|=\bar\al$
we would have 
$f``[\pi``\bar\al]^\omega=\al$, and so there would be some 
$s\in[\pi``\bar\al]^\omega$ such that $f(s)=\bar\ka$.
But now since $\omega<\bar\ka$, $s$ is of the form $\pi(t)$ for some
$t\in[\bar\al]^\omega$.
By elementarity, $\pi(\bar f(t))=f(\pi(t))$, so $\bar\ka$ is in the
range of $\pi$, a contradiction.
\end{proof}

In particular, if $\ka$ is the critical point of
a rank-plus-one-to-rank-plus-one embedding $j:V_{\la+1}\to V_{\la+1}$
then it is \emph{not} the case that for all $A\subseteq H_{\la^+}$, 
$j$ witnesses the loose $\la^+$-subcompactness of $\ka$ for $A$.
On the other hand, the existence of such an embedding does imply that
$\ka$ is loosely $\la$-subcompact.
\begin{prop}
Suppose $j:V_{\la+1}\to V_{\la+1}$ is an elementary embedding with
critical point $\ka<\la$.  Then $\ka$ is loosely $\la$-subcompact.
\end{prop}
\begin{proof}
It is a standard consequence of Kunen's inconsistency theorem
that the hypotheses imply that $\la=\sup_{n\in\omega}j^n(\ka)$
(see the proof of Proposition~\ref{loosefullorrankrank} for 
explicit details of essentially the same argument).
In particular, $\la$ is a limit of inaccessible cardinals,
so $V_\la=H_\la$ and $j\restr H_\la\in V_{\la+1}$.
Suppose $A\subseteq H_\la$.  Then in $V_{\la+1}$ we have
\begin{multline*}
\exists X\exists\pi:(H_\la,\in,X)\to(H_\la,\in,j(A))\\
(\pi\text{ is elementary}\land\pi(\crit(\pi))=j(\ka))
\end{multline*}
witnessed by $A$ and $j\restr H_\la$ as $X$ and $\pi$.  
Hence, pulling back by 
$j$ we have by elementarity that $V_{\ka+1}$ also satisfies
\begin{multline*}
\exists X\exists\pi:(H_\la,\in,X)\to(H_\la,\in,A)\\
(\pi\text{ is elementary}\land\pi(\crit(\pi))=\ka).
\end{multline*}
Since $V_{\la+1}$ is correct for this, we may conclude that $\ka$ is
loosely $\la$-subcompact.
\end{proof}

We do not know whether one can have a cardinal $\ka$
that is loosely $\al$-subcompact such that for some $A\subseteq H_\al$ the
only witnesses to loose $\al$-subcompactness of $\ka$ for $A$ are rank-to-rank
embeddings $\pi:H_\al\to H_\al$.
However, this situation is the only possible obstruction to the
equivalence of loose $\al$-subcompactness and full $\al$-subcompactness.

%
%
%

\begin{prop}\label{loosefullorrankrank}
If $\ka$ is loosely $\al$-subcompact, then $\ka$ is $\al$-sub\-com\-pact or
$\ka$ is the critical point of a rank-to-rank embedding $j:V_\al\to V_\al$.
\end{prop}
\begin{proof}
Suppose first that we have $\ka$ loosely $\al$-subcompact,
$A\subseteq H_{\al}$, and for every 
$\pi:(H_{\bar\al},\in,\bar A)$ $\to(H_\al,\in, A)$ witnessing loose 
$\al$-subcompactness of $\ka$ for $A$, $\bar\al=\al$.
For any such $\pi$, $\{\pi^n(\ka)\st n\in\omega\}$ is a subset of $H_{\bar\al}$ 
and has image $\{\pi^{n+1}(\ka)\st n\in\omega\}$ under $\pi$, 
so if we take $\la=\sup_{n\in\omega}(\pi^n(\ka))$,
then either $\la=\al$ or $\la$ is a fixed point of
$\pi$.  
By Lemma~\ref{barless}, $\cf(\al)=\omega$, 
and in particular $\al$ is a limit cardinal,
so if $\la<\al$ then $\la^{++}$ is also less than $\al$, 
and is a fixed point of $\pi$.
But then $\pi\restr H_{\la^{++}}$ is a non-trivial elementary embedding from
$H_{\la^{++}}$ to $H_{\la^{++}}$, violating Kunen's inconsistency result. 
Thus, we may assume that $\al=\la$ in all such cases, and so $H_\al=V_\al$.
As the image of the critical point $\bar\ka$ of $\pi:V_\al\to V_\al$,
it follows that $\ka$ is the critical point of
\[
\pi[\pi]=\bigcup_{\ga<\al}\pi(\pi\restr V_\ga)
\]
in the left self-distributive system of elementary embeddings 
from $V_\al$ to $V_\al$
(see for example \cite{Deh:EEA} for more on such embeddings).


It therefore only remains to consider the case in which
for each $A\subseteq H_\al$ there is a 
$\pi_A:(H_{\bar\al_A},\in,\bar A)\to(H_\al,\in,A)$ 
with critical point $\bar\ka_A$, witnessing
the loose $\al$-sub\-com\-pact\-ness of $\ka$ for $A$, 
such that $\bar\al_A<\al$.
For each $A$ take $\pi_A$ with $\bar\ka_A$ minimal, and with
$\bar\al_A$ minimal amongst those for our fixed $\bar\ka_A$.  
Then we claim that $\bar\al_A<\ka$.
For otherwise, we may use the fact that $\ka$ is $\bar\al_A$-subcompact.
More specifically,
the restriction of 
\[
\pi_{A,H_{\bar\al_A},\bar A,\{\bar\ka\}}:
(H_{\bar\al'},\in,\bar A',H_{\bar{\bar\al}_A},\bar{\bar A,}\,\{\bar{\bar\ka}\})
\to
(H_\al,\in,A,H_{\bar\al_A},\bar A,\{\bar\ka\})
\]
to $H_{\bar{\bar\al}_A}$
gives an elementary embedding 
\[
\rho:(H_{\bar{\bar\al}_A},\in,\bar{\bar A\;})\to
(H_{\bar\al_A},\in,\bar A),
\]
with critical point at least $\bar\ka$ by the minimality of $\bar\ka$.
Since $\rho(\bar{\bar\ka})=\bar\ka$, $\bar{\bar\ka}=\bar\ka$ and
$\crit(\rho)$ is in fact strictly greater than $\bar\ka$.
But also $\bar{\bar\al}<\bar\al$, so
$\pi_A\of\rho:(H_{\bar{\bar\al}_A},\in,\bar{\bar A\;})\to
(H_\al,\in,A)$
is an elementary embedding with critical point $\bar\ka$ witnessing the
$\al$-subcompactness of $\ka$ for $A$ with $\bar{\bar\al}<\bar\al$,
contradicting the choice of $\bar\al$.
\end{proof}

Now to the matter of the consistency strength of subcompactness
itself.  
It turns out that the levels of
subcompactness interleave with the levels of supercompactness in strength.
Indeed one gets a result much like Magidor's 
characterisation of supercompactness \cite{Mag:SEL}, 
just with $H_{\al}$ in place of $V_\al$ and the predicate $A$ added.
\begin{prop}
\begin{enumerate}
\item \label{supsub}
If $\ka$ is $2^{<\al}$-supercompact, then $\ka$ is $\al$-sub\-com\-pact.
\item \label{subsup}
If $\ka$ is $(2^{(\la^{<\ka})})^+$-subcompact, 
then $\ka$ is $\la$-supercompact.
\end{enumerate}
In particular, $\ka$ is supercompact if and only if $\ka$ is $\al$-subcompact
for every $\al>\ka$.
\end{prop}
We spell out the proof, appropriately modified from \cite{Mag:SEL}, 
for the sake of completeness.
\begin{proof}
\ref{supsub}.  Suppose $\ka$ is $2^{<\al}$-supercompact, and let
this be witnessed by
$j:V\to M$ with critical point $\ka$, $j(\ka)>2^{<\al}$, and
${}^{2^{<\al}}M\subset M$.
For every $A\subseteq H_\al$,
the restriction of $j$ to $H_{\al}$ is elementary from
$(H_{\al},\in,A)$ to $(H_{j(\al)}^M,\in,j(A))$, 
and since $|H_\al|=2^{<\al}$, $j\restr H_\al$ is a member of $M$.
Thus, with $\al$, $A$ and $j\restr H_\al$ as witnesses for the existential
quantifications, we have
\begin{multline*}
M\sat\exists\bar\al<j(\ka)\,\exists\bar A\subseteq H_{\bar\al}\,
\exists\pi:(H_{\bar\al},\in,\bar A)\to(H_{j(\al)},\in,j(A))\\
(\pi\text{ is an elementary embedding }\land\pi(\crit(\pi))=j(\ka)),
\end{multline*}
whence
\begin{multline*}
V\sat\exists\bar\al<\ka\,\exists\bar A\subseteq H_{\bar\al}\,
\exists\pi:(H_{\bar\al},\in,\bar A)\to(H_{\al},\in,A)\\
(\pi\text{ is an elementary embedding }\land\pi(\crit(\pi))=\ka).
\end{multline*}

\noindent\ref{subsup}.  Suppose $\ka$ is $(2^{\la^{<\ka}})^+$-subcompact, and 
let $\pi:(H_{\bar\al},\in,\{\bar\la\})\to(H_{(2^{\la^{<\ka}})^+},\in,\{\la\})$
witness this for the predicate $A=\{\la\}$,
with $\crit(\pi)=\bar\ka$.
By elementarity we have that $\bar\al=(2^{\bar\la^{<\bar\ka}})^+$,
and since $\bar\al<\ka$, we have in particular that $\bar\la<\ka$.

We claim that 
$\bar\ka$ is $\bar\la$-supercompact. 
To see this, define an ultrafilter $\calU$ on $\Power_{\bar\ka}\bar\la$ by
\[
X\in\calU\iff X\subseteq\Power_{\bar\ka}\bar\la\land
\pi(X)\ni\{\pi(\zeta)\st\zeta\in\bar\la\}.
\]
It is standard to check that $\calU$ so defined is a $\bar\ka$-complete
normal ultrafilter on $\Power_{\bar\ka}\bar\la$,
noting that $P_{\bar\kappa}\bar\la$ belongs to the domain of $\pi$ 
and $\pi(\bar\kappa) = \kappa$ is greater than $\bar\la$.
Now
$\calU\in H_{(2^{\bar\la^{<\bar\ka}})^+}$, and
\[
H_{(2^{\bar\la^{<\bar\ka}})^+}\sat\calU\text{ is a normal ultrafilter on }
\Power_{\bar\ka}\bar\la.
\]
Therefore by elementarity
\[
H_{(2^{\la^{<\ka}})^+}\sat\pi(\calU)\text{ is a normal ultrafilter on }
\Power_{\ka}\la,
\]
and $H_{(2^{\la^{<\ka}})^+}$ is clearly correct for this statement.
Hence, $\ka$ is $\la$-supercompact.
\end{proof}

Observe that the level of subcompactness required in (\ref{subsup}) 
to imply any supercompactness is at least $\ka^{++}$.
Indeed, Cummings and Schimmerling~\cite[Section~6]{CuS:ISq} note that
a $\ka^+$-subcompact cardinal $\ka$ need not be measurable, since
a measurable $\ka^+$-subcompact cardinal $\ka$ 
must have a normal measure 1 set
of $\iota^+$-subcompact cardinals $\iota$ below it.

\section{Squares}

The general proof of the incompatibility of subcompactness with $\square$ 
is similar to that for the $\ka^+$ case, 
due to Jensen.

\begin{thm}\label{subnotsq}
Suppose $\ka$ is $\al^+$-subcompact for some $\ka\leq\al$.  Then
$\square_\al$ fails.
\end{thm}
\begin{proof}
Suppose for contradiction that
$C=\langle C_\be\st\be\in\al^+\cap\Lim\rangle$ is a 
$\square_\al$-sequence.
We can take an $\al^+$-subcompactness embedding
\[
\pi:(H_{\bar\al^+},\in,\bar C)\to(H_{\al^+},\in,C)
\]
with critical point some $\bar\ka<\bar\al^+$, so that $\pi(\bar\ka)=\ka$.
Let $\la$ be the supremum of $\pi``(\bar\al^+)$, and consider 
$D=\lim(C_\la)\cap\pi``(\bar\al^+)$.  
Since $\pi``(\bar\al^+)$ is countably closed
(indeed, it contains all of its limits of cofinality less than $\bar\ka$) 
and unbounded in $\la$, $D$ is also unbounded in 
$\la$.  Therefore, since $\la$ has cofinality $\bar\al^+$, 
$D$ is a subset of the range of $\pi$ 
which has cardinality at least $\bar\al^+$, but order type less than $\al$,
by the definition of $\square_\al$.
For $\be_0<\be_1$ in $D$ we have that
$C_{\be_0}$ is an initial segment of $C_{\be_1}$ by coherence, and hence
$\ot(C_{\be_0})<\ot(C_{\be_1})<\al$.  But then
$\{\ot(C_\be)\st\be\in D\}$ must be a subset of the range of $\pi\restr\bar\al$,
and yet has cardinality $\bar\al^+$,
a contradiction.
\end{proof}

Now to the optimality of this result.

\begin{thm}\label{nosubsqforc}
Suppose the GCH holds, and 
let 
\[
I=\{\al\st\exists\ka\leq\al(\ka\text{ is $\al^+$-subcompact})\}.
\]
Then there is a cofinality-preserving partial order $\P$ such that
for any $\P$-generic $G$ the following hold.
\begin{enumerate}
\item $\square_\al$ holds in $V[G]$ for all $\al\notin I$.
\item If $\ka<\al$ are such that 
$V\sat\ka\text{ is $\al$-subcompact}$, then
\[
V[G]\sat\ka\text{ is $\al$-subcompact}.
\]
  In particular,
$I^{V[G]}=I$.
\end{enumerate}
\end{thm}
\begin{proof}
The partial order $\P$ will be a reverse Easton forcing iteration ---
see for example \cite{Cum:IFE} for an introduction to such forcings.
At stage $\al$ for $\al$ a cardinal not in $I$, we force with the usual
size $\al^{+}$ (thanks to the GCH),
$<\al^+$-strategically closed partial order 
$\bbS_\al$ due to Jensen to obtain
$\square_\al$, 
which uses initial segments of the generic $\square_\al$ sequence as conditions
--- see \cite[Example 6.3]{Cum:IFE}.
At all other stages we take the trivial forcing.
Thus, the iteration preserves cofinalities and the GCH, 
and $\square_\al$ holds in 
$V[G]$ for all $\al\notin I^V$.  It therefore only remains to show that
forcing with $\P$ preserves the $\al$-subcompactness of
any $\ka$ that is $\al$-subcompact in $V$.

So suppose $\ka$ is $\al$-subcompact in $V$.
By the definition of $I$, the forcing is trivial on the interval
$[\ka,\al)$.  Since every iterand from stage $\al$ onward is 
$<\al^{+}$-strategically closed, it follows that the tail of the iteration
starting at stage $\ka$ is $<\al^{+}$-strategically closed
--- see \cite[Proposition~7.8]{Cum:IFE}.
Hence, no new subsets of $\al$ are added by this part of the forcing.
By the GCH, $V\sat|H_\al|=\al$, and so the tail of the iteration starting at 
stage $\ka$ adds no new subsets of $H_\al$.
Thus,
to consider arbitrary subsets of $H_\al$ in the generic 
extension, it suffices to consider those of the form $\rho_G$ for
$\rho$ a $\P_\ka$-name,
where $\P_\ka$ denotes the
iteration of length $\ka$ that is the initial part of $\P$ up to
(but not including) $\ka$.  We shall denote by $G_\ka$ the generic
for $\P_\ka$ obtained from $G$, and use corresponding notation for $\bar\ka$.
Note in particular that $\rho$ can be taken to be a subset of $H_{\al}$.


Applying the $\al$-subcompactness of $\ka$ in $V$,
let 
\[
\pi:(H_{\bar\al},\in,\bar\rho)\to(H_{\al},\in,\rho)
\] 
witness the $\al$-subcompactness of $\ka$ for $\rho$, 
with critical point $\bar\ka$ taken by $\pi$ to $\ka$.
We wish to lift $\pi$ to an elementary embedding
$\pi':(H_{\bar\al}^{V[G]},\in,\bar\rho_G)\to
(H_{\al}^{V[G]},\in,\rho_G)$.
As noted in Section~\ref{prelims},
$\bar\ka$ is $<\bar\al$-subcompact if $\al$ is a limit cardinal
and $\bar\be$-subcompact if $\bar\al=\bar\be^+$, so in either case
$\P$ is trivial on the interval $[\bar\ka,\bar\al)$. 
Furthermore, even if the forcing iterand at stage $\bar\al$ is non-trivial,
it will be $<\bar\al^+$-strategically closed, and hence adds no new sets to 
$H_{\bar\al}$.  Indeed the tail of the forcing from stage 
$\bar\al$ on is $<\bar\al^+$-strategically closed.  Therefore
$H_{\bar\al}^{V[G]}=H_{\bar\al}^{V[G_{\bar\ka}]}$, so combining this with
$H_{\al}^{V[G]}=H_{\al}^{V[G_\ka]}$ our goal becomes to 
lift $\pi$ to 
\[
\pi':(H_{\bar\al}^{V[G_{\bar\ka}]},\in,\bar\rho_{G_{\bar\ka}})\to
(H_{\al}^{V[G_\ka]},\in,\rho_{G_\ka}),
\]
for which it suffices by the usual (Silver) argument to show that
$\pi``G_{\bar\ka}\subseteq G_{\ka}$.  But $\pi$ is the identity below $\bar\ka$,
so this is immediate.
\end{proof}

It should be noted that this lifting argument did not require that the generic
contain a non-trivial master condition.  
Hence, \emph{every} $\P$-generic $G$ over $V$
will preserve all $\be$-subcompacts for all $\be$.

Of course, it is important that our forcing preserve not only 
$\al$-subcompact cardinals, but stronger large cardinals too.
We verify this for a test case near the top of the large cardinal hierarchy,
specifically, $\omega$-superstrong cardinals 
(I2 cardinals in the terminology of \cite{Kan:THI}).  Recall their definition.

\begin{defn}\label{omss}
A cardinal $\ka$ is \emph{$\omega$-superstrong} if and only if there is
an elementary embedding $j:V\to M$ with critical point $\ka$ such that, 
if we let $\la=\sup_{n\in\omega}(j^n(\ka))$, $V_\la\subset M$.
\end{defn}

Note that we may take $M$ such that every element of $M$ has the form
$j(f)(a)$ for some $f$ with domain $V_\la$ and some $a\in V_\la$.
Indeed, given any $j:V\to N$ witnessing $\omega$-superstrength, the 
transitive collapse of the class of elements of this form gives such an $M$.

\begin{prop}\label{omsspressq}
The forcing iteration $\P$ of Theorem~\ref{nosubsqforc} preserves all
$\omega$-superstrong cardinals.
\end{prop}
\begin{proof}
We again use Silver's method of lifting embeddings.
Let $\ka$ be $\omega$-superstrong, let $j:V\to M$ witness this, 
let $\la$ be as in Definition~\ref{omss}, 
and suppose we have chosen $j$  
in such a way that every element of $M$ is of the form $j(f)(a)$ 
for some $a$ in $V_\la$ and $f$ with domain $V_\la$.
It follows from $\omega$-superstrength that $\ka$ is 
$\al$-subcompact for every $\al<\la$, that is, $<\la$-subcompact.
Thus, our forcing $\P$ is trivial between $\ka$ and $\la$. 
Also, since the definition of $I\cap V_\la$ is absolute for models 
containing $V_\la$, $j(\P_\la^V)=\P_\la^M=\P_\la^V$ (hence
the support of $\P$ will also be bounded below $\ka$).
Below $\la$, therefore, we may just take the generic for $M$ to be the generic
for $V$, $G_\la$.
The embedding $j$ is the identity on nontrivial stages of the iteration,
so $j``G_\la$ is trivially a subset of $G_\la$, and we get a lift $j'$ of $j$ 
from $V[G_\la]$ to $M[G_\la]$.  

We claim that for the tail of the forcing, 
the pointwise image of the tail of the generic for $V$,
$j'``G^\la$, generates a generic filter for $M$,
by the $\la^+$-distributivity of this tail forcing.
Indeed this is standard for preservation results about $\omega$-superstrongs:
compare for example with \cite{SDF:LCL} and \cite{Me:LCDWO}.
To be explicit: every element of $M[G_\la]$ is of the form 
$\sigma_{G_\la}$ for some $\sigma=j(f)(a)$ with $a\in V_\la$.
Suppose $D$ is a dense class in the tail of the forcing iteration,
defined in $M[G_\la]$ as $\{p\st\psi(p,d)\}$ for some parameter 
$d=j(f)(a)_{G_\la}$ with $a\in V_\la$.
Since the tail $\P^\la$ of the forcing is $<\la^+$-strategically closed
and $|V_\la|=\la$,
it is dense for $q\in\P^{\la}$ to extend an element of
$D_x=\{p\st\psi(p,f(x)_{G_\la})\}$
whenever $x\in V_\la$ and $D_x$ is dense in $\P^\la$.
We may therefore take such a $q$ lying in $G^\la$, and by elementarity 
have that $j(q)$ extends $D$.
That is, $j'``G^\la$ indeed generates a generic filter over $M$ for
$(\P^\la)^M$.
\end{proof}

\section{Stationary reflection}

For simplicity, we restrict attention to cofinality $\omega$.
This is to ensure that the cofinality of interest is not affected by the 
embeddings involved --- any small enough cofinality would suffice.
We also stick with stationary subsets of $\al^+$, although variants such as
stationary subsets of $[\al^+]^{\aleph_0}$ would also be interesting to
consider in this context.

As noted after Definition~\ref{SR}, $\SR(\al^+,\omega)$ can be seen as a
stronger ``compactness phenomenon'' than 
the failure of square.  Correspondingly, we consider a strengthening
of subcompactness.

\begin{defn}\label{statsub}
For any cardinal $\al$, we say that 
a cardinal $\ka\leq\al^+$ is \emph{$(\al^+,\omega)$-stationary subcompact} if 
for every $A\subseteq H_{\al^+}$ and 
every stationary set
$S\subseteq {\al^+}\cap\Cof(\omega)$, 
there exist $\bar\ka<\bar\al^+<\ka$, 
$\bar A\subseteq H_{\bar\al^+}$, 
a stationary set $\bar S\subseteq\bar\al^+\cap\Cof(\omega)$
and an elementary embedding 
\[
\pi:(H_{\bar{\al}^+},\in,\bar A,\bar S)\to(H_{\al^+},\in,A,S)
\]
with critical point $\bar\ka$ such that $\pi(\bar\ka)=\ka$.
We say that such an embedding $\pi$
\emph{witnesses the $(\al^+,\omega)$-stationary subcompactness of $\ka$ for $A$
and $S$}.
\end{defn}
Thus, $({\al^+},\omega)$-stationary subcompactness 
is $\al^+$-subcompactness with the
extra requirement that there be witnessing embeddings respecting the 
stationarity of any given $S\subseteq\al^+\cap\Cof(\omega)$.
As for sub\-com\-pact\-ness, we can and will freely replace $A$ in the
definition by any
finite number of subsets of $H_{\al^+}$.
Since $H_{\ga^+}$ is correct for stationarity of subsets of $\ga$,
we have that if $\ka<\be^+<\al$ and $\ka$ is $\al$-subcompact, then
$\ka$ is $(\be^+,\omega)$-stationary subcompact, and moreover if
$\pi:(H_{\bar\al},\in,\bar A)\to(H_{\al},\in,A)$ is an embedding
with critical point $\bar\ka$
witnessing $\al$-subcompactness of $\ka$ for any $A\subseteq H_\al$, then
for all $\bar\be^+<\bar\al$, 
$\bar\ka$ is $(\bar\be^+,\omega)$-stationary subcompact.

This strengthened subcompactness notion 
is sufficient to obtain stationary reflection as a 
consequence.

\begin{prop}
If there exists some $\ka\leq\al$ such that 
$\ka$ is $(\al^{+},\omega)$-stationary subcompact,
then $\SR(\al^+,\omega)$ holds.
\end{prop}
\begin{proof}
Suppose $\ka\leq\al$ is $(\al^{+},\omega)$-stationary subcompact, 
let $S$ be a stationary subset of $\al^+\cap\Cof(\omega)$,
take $A\subseteq H_{\al^+}$ arbitrary,
and let $\pi:(H_{\bar\al^{+}},\in,\bar A,\bar S)\to$
\mbox{$(H_{\al^{+}},\in,A,S)$}
with critical point $\bar\ka$ 
witness $(\al^{+},\omega)$-stationary sub\-com\-pact\-ness of $\ka$ for 
$A$ and $S$.
Let $\la=\sup(\pi``\bar\al^+)$;
we claim that $S\cap\la$ is stationary in $\la$.
The pointwise
image of $\bar\al^+$ in $\al^+$ is countably closed and unbounded in $\la$,
so for any club $C\subseteq\la$, \mbox{$C\cap\pi``\bar\al^+$} is also
countably closed and unbounded in $\la$.
Therefore, $\pi^{-1}C$ is countably closed and unbounded in $\bar\al^+$,
and hence has nonempty intersection with $\bar S$.
But now taking $\xi\in\bar S\cap\pi^{-1}C$, 
we have
$\pi(\xi)\in S\cap C$.  Hence, $S\cap\la$ is stationary.
\end{proof}

Again, we have a complementary result under the GCH.

\begin{thm}\label{nosubSRforc}
Suppose the GCH holds.
Let $I$ be as defined in Theorem~\ref{nosubsqforc}, and similarly let
\[
J=\{\al\st\exists\ka\leq\al(\ka
\text{ is $(\al^{+},\omega)$-stationary subcompact})\}
\subseteq I.
\]
Then there is a cofinality-preserving partial order $\P$ such that
for any $\P$-generic $G$ the following hold.
\begin{enumerate}
\item\label{SRforcSR} 
$\SR(\al^+,\omega)$ fails in $V[G]$ for all $\al\notin J$.
\item\label{SRforcSq} $\square_\al$ holds in $V[G]$ for all $\al\notin I$.
\item\label{SRforcJpres} If $\ka\leq\al$ are such that 
$V\sat\ka\text{ is $(\al^+,\omega)$-stationary subcompact}$,
then
$V[G]\sat\ka\text{ is $(\al^+,\omega)$-stationary subcompact}$.
In particular,
$J^{V[G]}=J$.
\item\label{SRforcIpres} $I^{V[G]}=I$.
\end{enumerate}
\end{thm}
Note that article \ref{SRforcIpres} is a weaker statement than the
natural analogue of article \ref{SRforcJpres}.  Because of problems that
could potentially arise from ``overlapping'' embeddings, we do not
claim to preserve every instance of subcompactness, 
but we can preserve sufficiently many of them
for $I$ to remain unchanged.
\begin{proof}
Again $\P$ will be a reverse Easton iteration.
At stage $\al$ for $\al\in J$, we take the trivial forcing.
For cardinals $\al\in I\smallsetminus J$, we take the forcing $\R_{\al}$
that adds a 
non-reflecting stationary set to $\al^+\cap\Cof(\omega)$ by initial segments; 
this 
forcing is $\al^+$-strategically closed and (by the GCH) 
of size $\al^+$
(see \cite[Example~6.2]{Cum:IFE}).
For cardinals $\al\notin I$, we take a three stage iteration, first forcing 
with $\R_\al$.
Next, we force
with the partial order $\C_\al^\R$ that makes the generic stationary set from
$\R_\al$ non-stationary by shooting a club through its complement.
Third, we force to make $\square_\al$ hold with $\bbS_\al$.
The two stage iteration $\R_\al*\dot\C_\al^\R$ 
is $<\al^+$-strategically closed
(indeed it contains a natural dense sub\-order that is $<\al^+$-closed), 
and $\bbS_\al$
is also $<\al^+$-strategically closed, so 
$\R_\al*\dot\C_\al^\R*\dot\bbS_\al$ is $<\al^+$-strategically closed.
It also has a dense suborder of size $\al^+$.
Thus, our reverse Easton iteration will indeed preserve cofinalities, as
well as the GCH.
The generic extension 
will also clearly satisfy \ref{SRforcSR} and \ref{SRforcSq} of the theorem.

As in the proof of Theorem~\ref{nosubsqforc},
we will denote by $\P_\ka$ the iteration up to stage $\ka$ 
and by $G_\ka$ the corresponding generic; 
note that if $\ka$ is inaccessible, $\P_\ka$ is a 
direct limit, so we can and will identify
$\P_\ka$ with $\bigcup_{\ga<\ka}\P_\ga$.


If $\ka$ is $(\al^+,\omega)$-stationary subcompact, 
then the forcing is trivial in
stages from $\ka$ up to (but not necessarily including) $\al^+$, 
and is $<\al^{++}$-strategically closed from stage $\al^+$ onward,
so no new elements or subsets of $H_{\al^+}$ are added after stage $\ka$.
Thus, to show that $(\al^+,\omega)$-stationary subcompactness is preserved, 
it suffices
to show that for any $\P_\ka$-name $\rho$ for a subset of
$H_{\al^+}^{V[G]}$ and any $\P_\ka$-name $\sigma$ for a stationary subset of 
$\al^+$, there is an embedding
from \mbox{$(H_{\bar\al^+}^{V[G]},\in,\bar\rho_G,\bar\sigma_G)$} 
to $(H_{\al^+}^{V[G]},\in,\rho_G,\sigma_G)$
witnessing the $\al^+$-stationary subcompactness of $\ka$ for $\rho_G$ and
$\sigma_G$ in $V[G]$.

Because $\P_\ka$ is
only of cardinality $\ka$, 
there is some $p\in G$ and some $S\in V$ stationary in $\al^+$
such that 
$p\forces \check S\subseteq\sigma$, and $H_{\al^+}$ contains all the requisite
sets to be correct for this statement.
In $V$ let
$\pi:(H_{\bar\al^+},\in,\bar\rho,\bar\sigma,\bar p, \bar S)\to
(H_{\al^+},\in,\rho,\sigma, p, S)$
with critical point $\bar\ka$ and $\bar S$ stationary in $\bar\al^+$
witness $\al^+$-stationary subcompactness of $\ka$ for 
$\rho, \sigma, p$ and $S$.
Then by elementarity, $\bar p\forces\check{\bar S}\subseteq\bar\sigma$,
and moreover, $\bar p$ is a condition bounded below $\bar\ka$,
so since $\bar\ka=\crit(\pi)$, $\bar p=\pi(\bar p)=p$.
It follows, since $\P_{\bar\ka}$ is small relative to 
$\bar\al^+$, that $\bar S$ remains stationary under forcing with $\P_{\bar\ka}$,
and so $p$ forces $\bar\sigma$ to be stationary in $\bar\al^+$.
Now by Silver's lifting of embeddings method again,
$\pi$ lifts to an elementary embedding
$\pi':(H_{\bar\al^+}^{V[G_{\bar\ka}]},\in,
\bar\rho_{G_{\bar\ka}},\bar\sigma_{G_{\bar\ka}})\to
(H_{\al^+}^{V[G_\ka]},\in,\rho_{G_\ka},\sigma_{G_\ka})$, since
$\pi``G_{\bar\ka}=G_{\bar\ka}$.
That is, we have
$\pi':(H_{\bar\al^+}^{V[G]},\in,\bar\sigma_{G})\to
(H_{\al^+}^{V[G]},\in,\sigma_{G})$ with $\bar\sigma_G$ stationary, as
required.



To prove part~\ref{SRforcIpres}, it now suffices to consider the case when
$\ka$ is $\al^+$-subcompact but no $\ka'$ is 
$(\al^+,\omega)$-stationary subcompact.
Furthermore, in order to show that $I$ is preserved, it suffices to only show
that $\al^+$-subcompactness of $\ka$ is preserved when $\ka$ is the
\emph{least} $\al^+$-subcompact cardinal.
So suppose that $\ka$ is the least $\al^+$-subcompact cardinal and no $\ka'$ is
$(\al^+,\omega)$-stationary subcompact.

For any $A\subseteq H_{\al^+}$, we claim there is a 
$\pi:(H_{\bar\al^+},\in,\bar A)\to\,$\mbox{$(H_{\al^+},\in,A)$} 
witnessing the $\al^+$-subcompactness of $\ka$ for $A$ such that 
$\bar\al\notin I$, that is, no $\ka'$ is $\bar\al^+$-subcompact.
To see this, let $B$ be a subset of \mbox{$H_{\al^+}\times\ka$}
$\,\subset H_{\al^+}$
such that for each cardinal $\ga<\ka$, the cross-section
$B_\ga=\{$\mbox{$x\in H_{\al^+}$}$\,\st(x,\ga)\in B\}$ witnesses the failure of
$\ga$ to be $\al^+$-subcompact, that is, there is no embedding $\pi'$
witnessing $\al^+$-subcompactness of $\ga$ for $B_\ga$.
Let $\pi:(H_{\bar\al^+},\in,\bar A,\bar B)\to(H_{\al^+},\in,A,B)$ be an 
embedding witnessing the $\al^+$-subcompactness of $\ka$ for $A$ and $B$
with minimal critical point, and given the critical point, minimal $\bar\al$.
Call the critical point $\bar\ka$ as always.
We claim that $\pi$ considered as an embedding 
from $(H_{\bar\al^+},\in,\bar A)$
to $(H_{\al^+},\in,A)$
is as required.
If $\bar\ka$ itself were $\bar\al^+$-subcompact, there would be an elementary
embedding 
$\phi:\,$\mbox{$(H_{\bar{\bar\al}^+},\in,\bar{\bar A,}\,\bar{\bar B\,})$}
$\,\to\,$
\mbox{$(H_{\bar\al}^+,\in,\bar A,\bar B)$}
witnessing the $\bar\al^+$-subcompactness of $\bar\ka$ for $\bar A$
and $\bar B$, and then $\pi\of\phi$ would be an embedding witnessing the 
$\al^+$-subcompactness of $\ka$ for $A$ and $B$ with critical point less than
$\bar\ka$, violating the choice of $\pi$.
Similarly if some $\ka'>\bar\ka$ were $\bar\al^+$-subcompact, 
then there would be
an elementary embedding
$\phi:\,$\mbox{$(H_{\bar{\bar\al}^+},\in,\bar{\bar A,}\,\bar{\bar B\,})$}$\,\to
(H_{\bar\al}^+,\in,\bar A,\bar B)$
with critical point $\bar{\ka'}>\bar\ka$
witnessing the $\bar\al^+$-subcompactness of $\ka'$ for $\bar A$ and $\bar B$.
In this case, $\pi\of\phi$ would be an elementary embedding
witnessing the $\al^+$-subcompactness of $\ka$ for $A$ and $B$, 
with critical point
$\bar\ka$ but with domain $H_{\bar{\bar\al}^+}$ for some
$\bar{\bar\al}^+<\bar\al^+$, again violating the choice of $\pi$.
Finally, if some $\ka'<\bar\ka$ were $\bar\al^+$-subcompact,
there would be an elementary embedding
$\phi:\,$\mbox{$(H_{\bar{\bar\al}^+},\in,\bar{\bar A,}\,\bar{\bar B\,})$}$\,\to
(H_{\bar\al}^+,\in,\bar A,\bar B)$ with critical point $\bar{\ka'}$ witnessing
the $\bar\al^+$-subcompactness of $\ka'$ for $\bar A$ and $\bar B$.
But then $\pi\of\phi$ witnesses the $\al^+$-subcompactness of 
$\pi(\ka')=\ka'$ for $A$ and $B$, and in particular, we may view
$\pi\of\phi$ as an elementary embedding 
\mbox{$(H_{\bar{\bar\al}^+},\in,\bar{\bar B\,}\!_{\bar{\ka'}})$}$\,\to
(H_{\al^+},\in,B_{\ka'})$ witnessing
the $\al^+$-subcompactness of $\ka'$ for $B_{\ka'}$.
This of course violates the choice of $B$, and so the claim is proven.

Returning to the proof of part~\ref{SRforcIpres},
we have $\al$ and $\ka$ such that 
$\ka$ is the least $\al^+$-subcompact cardinal
and no $\ka'$ is $(\al^+,\omega)$-stationary subcompact,
and we wish to show that $\ka$ remains $\al$-subcompact in the
generic extension $V[G]$.
The forcing $\P$ is 
trivial on $[\ka,\al)$, is $\R_\al$ at stage $\al$, and 
is $<\al^{++}$-strategically closed thereafter.
Thus, $H_{\al^+}^{V[G]}=H_{\al^+}^{V[G_\ka]}$, and any subset of
$H_{\al^+}^{V[G]}$ is named by a $\P_{\al+1}\cong\P_{\ka}*\dot\R_\al$-name
which is a subset of $H_{\al^+}$;
of course, any such name is forced by $\mathds{1}_{\P_{\ka}*\dot\R_\al}$
to be a subset of $H_{\al^+}$.

So suppose 
$\sigma$ is such a name; we wish to lift an embedding in $V$ witnessing the
$\al^+$-subcompactness of $\ka$ for $\sigma$ to an embedding in $V[G]$
witnessing the $\al^+$-subcompactness of $\ka$ for $\sigma_G$.
Often such lifting arguments simply require one to find an appropriate
\emph{master condition} (see for example \cite[Section~12]{Cum:IFE}), 
as there will always be a generic including any condition.
However, we wish to lift embeddings for many different names $\sigma$, 
and it is not clear that the corresponding conditions can all lie in a common
generic.
One fix that is sometimes possible is to use homogeneity of the partial order
to argue that the generic can be modified to contain the master condition
without altering the genric extension it produces --- this is the 
approach taken in \cite{MeSDF:LCMor}, for example.  But again this is not
appropriate in the present context, as when $G$ is modified to give some $G'$,
the interpretation of the name $\sigma$ may be changed.
Instead, we shall show that master conditions for witnessing embeddings are 
\emph{dense} in the partial order, thus guaranteeing that for each $\sigma$
there is a corresponding master condition in 
any given $G$.
A similar technique has
been used to demonstrate the indestructibility of 
Vop\v{e}nka's Principle relative to many natural forcing
iterations \cite{Me:IVP}.

To this end, let $p$ be an arbitrary element of $\P_{\al+1}$, and
for each $\ga<\ka$
take 
\[
\pi_\ga:(H_{\bar\al_\ga^+},\in,\bar\sigma_\ga,\{\bar p_\ga\}, \{\ga\})\to
(H_{\al^+},\in,\sigma,\{p\}, \{\ga\})
\] 
with critical point $\bar\ka_\ga$
witnessing $\al^+$-sub\-com\-pact\-ness of $\ka$ for $\sigma$,
$\{p\}$ and $\{\ga\}$. 
As shown above, 
we may assume that no $\ka'$ is 
$\bar\al_\ga^+$-subcompact, 
but by the comment after Definition~\ref{statsub}, $\bar\ka_\ga$ will
be 
$(\bar\be^+,\omega)$-stationary subcompact
for all $\bar\be<\bar\al_\ga$.
Thus, the forcing is trivial on $[\bar\ka_\ga,\bar\al_\ga)$,
at stage $\bar\al_\ga$ is 
$\R_{\bar\al_\ga}*\dot\C_{\bar\al_\ga}^\R*\dot\bbS_{\bar\al_\ga}$,
and is $<\bar\al_\ga^{++}$-strategically closed thereafter.
In particular, $H_{\bar\al_\ga^+}$ receives no new elements from stage
$\bar\ka_\ga$ of the forcing onward.

Note that by elementarity, 
$\bar\sigma$ is a $\P_{\bar\ka_\ga}*\dot\R_{\bar\al_\ga}$-name for a subset
of $H_{\bar\al_\ga^+}$.
Similarly, 
$p$ is comprised of
a $\P_\ka$-condition $p\restr\ka$ that is thus bounded below $\ka$,
and a name $\dot p(\al)$ for an $\R_\al$-condition,
so $\bar p_\ga$ consists of a $\P_{\bar\ka_\ga}$-condition 
$\bar p_\ga\restr\bar\ka_\ga$ and a name
$\dot{\bar p}_\ga(\bar\al_\ga)$ for an $\R_{\bar\al_\ga}$-condition.
Let $\ka'=\dom(p\restr\ka)$; 
then $\ka'$ is an ordinal less than $\ka$ in the range of $\pi$, 
so the critical point $\bar\ka_\ga$ of $\pi_\ga$ must be greater than $\ka'$,
and we have $\bar p_\ga\restr\bar\ka_\ga=p\restr\ka=p\restr\ka'$.
Since a direct limit is taken at $\ka$, 
every extension $q$ of $p\restr\ka'$ in $\P_\ka$ also has support bounded below
$\ka$, so if we take $\ga$ greater than this support, then
$\bar\al_\ga>\bar\ka_\ga>\ga$ is greater than the support of $q$,
and so $\bar p_\ga=p\restr\ka*\dot{\bar p}_\ga(\bar\al_\ga)$ is compatible
with $q$.
So we have shown that it is dense below $p\restr\ka$ to extend the
$\bar p_\ga$ for some $\pi_\ga$.

For terminological convenience 
we now move to
$V[G_{\ka}]$, still considering an arbitrary condition $p$ in 
$\P_{\al+1}$. 
Take $\pi_\ga$ as above such that $\bar p_\ga\in G_\ka$;
since it is now redundant, we henceforth drop $\ga$ from $\pi_\ga$
and all related notation.
Note that since $\pi$ is the identity on $V_{\bar\ka}$,
it lifts in the usual (Silver) way to an embedding
$\pi':H_{\bar\al^+}^{V[G_{\bar\ka}]}\to H_{\al^+}^{V[G_\ka]}$.
Let $G_{\R_{\bar\al}}$ denote the $\R_{\bar\al}$
generic over $V[G_{\bar\ka}]$ that comes from $G_\ka$.
Now, 
$r=\bigcup\pi'``G_{\R_{\bar\al}}$ is a condition in $\R_\al^{V[G_\ka]}$,
since the union of the pointwise image of 
the $\C_{\bar\al}^\R$-generic component of $G_\ka$ 
is a club in the complement 
of $r$ in 
$\sup(\pi'``\bar\al^+)$.
Since $\bar p(\bar\al)\in G_{\R_{\bar\al}}$, $r\leq p$,
and if $r\in G$, 
$\pi'$ lifts to an embedding
\[
\pi'':(H_{\bar\al^+}^{V[G_{\bar\ka}*G_{\R_{\bar\al}}]},\in,
\bar\sigma_{G_{\bar\ka}*G_{\R_{\bar\al}}})
\to
(H_{\al^+}^{V[G_{\ka}*G_{\R_{\al}}]},\in,
\sigma_{G_{\ka}*G_{\R_{\al}}}).
\]
But this is the same as
\[
\pi'':(H_{\bar\al^+}^{V[G]},\in,
\bar\sigma_{G})
\to
(H_{\al^+}^{V[G]},\in,
\sigma_{G}).
\]
Therefore, it is indeed dense to have such an embedding.
\end{proof}

Our forcing for Theorem~\ref{nosubSRforc} also 
preserves stronger large cardinals.

\begin{prop}
The forcing iteration $\P$ of Theorem~\ref{nosubSRforc} preserves all
$\omega$-superstrong cardinals
\end{prop}
The proof is exactly as for Proposition~\ref{omsspressq}.

\section{Weaker Squares}\label{weakersquares}

Schimmerling~\cite{Sch:CPCM1W} introduced the following generalisation of
$\square_\al$.
\begin{defn}
For any cardinal $\al$, a \emph{$\square_{\al,<\mu}$-sequence} is a sequence
$\langle\mathcal{C}_\be\st\be\in\al^+\cap\Lim\rangle$ such that for every 
$\be\in\al^+\cap\Lim$,
\begin{itemize}
\item $\mathcal{C}_\be$ is a set of closed unbounded subsets of $\be$,
\item $1\leq|\mathcal{C}_\be|<\mu$,
\item $\ot(C)\leq\al$ for every $C\in\mathcal{C}_\be$,
\item for any $C\in\mathcal{C}_\be$ and $\ga\in\lim(C)$, 
$C\cap\ga\in\mathcal{C}_\ga$.
\end{itemize}
We say $\square_{\al,<\mu}$ holds if there exists a 
$\square_{\al,<\mu}$-sequence, and we write
$\square_{\al,\nu}$ for $\square_{\al,<\nu^+}$.
\end{defn}

Of course, $\square_{\al,1}$ is simply $\square_\al$, and the
strength of the statement $\square_{\al,<\mu}$ is non-increasing as
$\mu$ increases; moreover Jensen \cite{Jen:SBOP} has shown that 
$\square_{\al,2}$ does not imply $\square_{\al,1}$.  
Jensen's \emph{weak square}, $\square^*_\al$, is 
simply $\square_{\al,\al}$, and $\square_{\al,\al^+}$ is provable in ZFC 
for all $\al$.




It turns out that some of these weaker forms of square are also precluded by 
$\al^+$-subcompactness of some $\ka<\al$.  
Indeed, 
corresponding results are known for $\ka$ an $\al^+$-supercompact cardinal,
so this should not be surprising.

\begin{thm}\label{subnosqalltcfal}
Suppose $\ka$ is $\al^+$-subcompact for some $\ka\leq\al$.  Then
$\square_{\al,<\cf(\al)}$ fails.
\end{thm}
\begin{proof}
Suppose for contradiction that
$\mathcal{C}=\langle\calC_\be\st\be\in\al^+\cap\Lim\rangle$ is a 
$\square_{\al,<\cf(\al)}$-sequence.
Note that clubs of order type $\al$ only occur at
ordinals with cofinality $\cf(\al)$.
We can take an $\al^+$-subcompactness embedding
\[
\pi:(H_{\bar\al^+},\in,\bar\calC)\to(H_{\al^+},\in,\calC)
\]
with critical point some $\bar\ka<\bar\al^+$ such that $\pi(\bar\ka)=\ka$,
and $\bar\al<\al$.
Let $\la$ be the supremum of $\pi``(\bar\al^+)$, 
let $C$ be an arbitrary member of $\calC_\la$, and consider 
the inverse image $\bar D$ of $\lim(C)$ under $\pi$.
Because $\pi``(\bar\al^+)$ is $<\bar\ka$-closed and unbounded in $\la$,
$\bar D$ is $<\bar\ka$-closed and unbounded in $\bar\al^+$, so we may take some 
$\bar\be\in\bar D$ of cofinality different from $\cf(\bar\al)$ such that 
$|\bar D\cap\bar\be|=\bar\al$.

Now, for any 
$\bar\ga<\bar\be$ in $\bar D$, $\pi(\bar\ga)\in C\cap\be\in\calC_\be$,
so by elementarity there is some $\bar C\in\bar\calC_{\bar\be}$ with
$\bar\ga\in\bar C$.  But there are fewer than $\cf(\bar\al)$ elements of 
$\bar\calC_{\bar\be}$, each of order type strictly less that $\bar\al$,
so $|\bigcup\bar\calC_{\bar\be}|<\bar\al$, 
and not all $\ga\in\bar D\cap\bar\be$ can be
covered in this way.
\end{proof}

Note that under the GCH, $\square_\al^*$ holds for all regular $\al$
(we may take \emph{all} clubs of order type less than $\al$ at ordinals of
cofinality less than $\al$), making
Theorem~\ref{subnosqalltcfal} optimal.
For singular $\al$, 
we leave obtaining a forcing reversal of the result 
until we have considered
obstructions to even weaker variants of $\square$.

Foreman and Magidor \cite{FMg:VWSP} observed 
that if $\square^*_{\al}$ holds 
then there is a $\square^*_{\al}$
sequence 
(referred to in \cite{CFM:SSSR} as an \emph{improved} square sequence,
$\square^{\text{imp}}_{\al,\al}$)
with the added property that for all $\be<\al^+$, there is a 
$C\in\calC_\be$ with $\ot(C)=\cf(\be)$.  Indeed, if we choose an 
arbitrary sequence $\langle D_\ga\st \ga\in\Lim\cap\al+1\rangle$ 
such that $D_\ga$ is a club in $\ga$ of order type
$\cf(\ga)$, then for any $\square^*_\al$-sequence $\calC$, we may obtain 
a $\square^{\text{imp}}_{\al,\al}$-sequence by adding
$\{\de\in C\st \ot(C\cap\de)\in D_\ga\}$ to 
$\calC_\be$ for every $C\in\calC_\be$ and $\ga$ such that
$\ot(C)\in\Lim(D_\ga)\cup\{\ga\}$.
Using this fact with a trick due to Solovay, we see that 
if there is some $\ka>\cf(\al)$ that is $\al^+$-subcompact, then
even  $\square^*_{\al}$ fails.

\begin{thm}\label{kagtcfnotwsq}
Suppose $\ka$ is $\al^+$-subcompact for some $\ka\leq\al$
with $\ka>\cf(\al)$.  Then $\square_{\al,\al}$ fails.
\end{thm}
\begin{proof}
We essentially follow 
the proof for the analogous result with $\ka$ $\al^+$-supercompact
due to Shelah~\cite{She108:SSC}
as presented by Cummings~\cite[Section~6]{Cum:NSCC}.
Suppose for contradiction that $\calC$ is a $\square^{\text{imp}}_{\al,\al}$
sequence, and let
\[
\pi:(H_{\bar\al^+},\in,\bar\calC)\to(H_{\al^+},\in,\calC)
\]
be an embedding witnessing the $\al^+$-subcompactness of $\ka$ for $\calC$.
Since $\cf(\al)=\pi(\cf(\bar\al))$, it is in particular in the range of $\pi$,
and hence $\cf(\al)<\ka$ implies that in fact $\cf(\al)<\bar\ka$.
Let $\la=\sup(\pi``\bar\al^+)$, and take $C\in C_\la$ with 
$\ot(C)=\bar\al^+=\cf(\la)$.  Let $\bar D$ be the preimage of $C$ under $\pi$;
as usual, it is a $<\bar\ka$-closed unbounded subset of $\bar\al^+$.  
Let $\zeta$ be the $\bar\al$-th element of $\bar D$.  
Since $\cf(\bar\al)<\bar\ka$, $\pi$ is continuous at $\zeta$, and in
particular $\pi(\zeta)$ is a limit point of $C$.
Thus, $C\cap\pi(\zeta)\in\calC_{\pi(\zeta)}$.
Now for every subset $X$ of $\bar D\cap\zeta$ of size less than $\bar\ka$,
$\pi(X)=\pi``X\subset C\cap\pi(\zeta)\in\calC_{\pi(\zeta)}$, 
so by elementarity, there is an element $\bar C_X$ of $\bar\calC_\zeta$ 
of order type less than $\bar\ka$ such 
that $X\subseteq\bar C_X$.
But there are $\bar\al^{<\bar\ka}>\bar\al$ such subsets $X$ of 
$\bar D\cap\zeta$ 
and at most $\bar\al$ such elements of $\bar\calC_\zeta$, each with
at most $2^{<\bar\ka}=\bar\ka<\bar\al$ subsets, 
yielding a contradiction.
\end{proof}

For readers familiar with scales, we note as an aside
that Theorem~\ref{kagtcfnotwsq}
is understating the case.  Indeed, as pointed out to us by the 
anonymous referee, we have the following.

\begin{thm}
Suppose $\ka$ is $\al^+$-subcompact for some $\ka\leq\al$ with $\ka>\cf(\al)$.
Let $\langle\al_i\st i<\cf(\al)\rangle$ be an increasing sequence of
regular cardinals cofinal in $\al$ with $\al_0>\ka$, and let
$f=\langle f_\ga\st\ga<\al^+\rangle$ be a scale in $\prod_i\al_i$.
Then $f$ fails to be good at stationarily many points in $\al^+$.
\end{thm}
\begin{proof}
Again, this follows from a straightforward modification of the proof
of the corresponding result for $\al^+$-supercompactness due to Shelah
\cite{She:CA}, as presented in \cite[Theorem~18.1]{Cum:NSCC}.
Indeed, if $\pi$ witnesses the $\al^+$-subcompactness of $\ka$ for $f$
and an arbitrary club $C$ in $\al^+$, then $\sup(\pi``\bar\al^+)$ is
a point of $C$ which is not a good point of $f$, as
$h:i\mapsto\sup(\pi``\bar\al_i)$ is an exact upper bound.
\end{proof}

As for our earlier results, we show by forcing that under the GCH, 
Theorems \ref{subnosqalltcfal} and \ref{kagtcfnotwsq} are in
some sense optimal.  

\begin{thm}\label{wksqforc}
Suppose the GCH holds.
Let $I$ be as defined in Theorem~\ref{nosubsqforc}, and similarly let
\[
K=\{\al\st\exists\ka>\cf(\al)(\ka
\text{ is $\al^{+}$-subcompact})\}
\subseteq I.
\]
Then there is a cofinality-preserving partial order $\P$ such that
for any $\P$-generic $G$ the following hold.
\begin{enumerate}
\item\label{wsforcSq} $\square_\al$ holds in $V[G]$ for all $\al\notin I$.
\item\label{wsforcws} $\square_{\al,\cf(\al)}$ holds in $V[G]$ for all
$\al\notin K$.
\item\label{wsforcIpres} $I^{V[G]}=I$.
\item\label{wsforcKpres} $K^{V[G]}=K$.
\end{enumerate}
\end{thm}
\begin{proof}
Once more we use a reverse Easton iteration $\P$.
In our iterands, 
we use the forcing partial order of Cummings, Foreman and 
Magidor~\cite[Theorem~16]{CFM:SSSR}
to force $\square_{\al,\cf(\al)}$ for singular $\al$.
We denote this partial order by $\bbT_\al$; note that it is
$<\cf(\al)$ directed closed and $<\al$-strategically closed,
and by the GCH has cardinality $\al^+$.

For regular cardinals $\al\notin I$, we force with $\bbS_\al$ at stage $\al$.
For singular cardinals $\al\in I\smallsetminus K$, we force with
$\bbT_\al$.
For singular cardinals $\al\notin I$, we force with the two-stage iteration
$\bbT_\al*\dot\bbS_\al$.
At all other stages we use the trivial forcing.
Clearly this gives a generic extension that satisfies 
\ref{wsforcSq} and \ref{wsforcws}, so we turn to preservation of $I$ and $K$.

Because of \ref{wsforcSq}, \ref{wsforcws}, and 
the fact that cofinalities are preserved, it suffices to lift various 
embeddings witnessing $\al^+$-subcompactness. 
Indeed, there are three cases for which we need to check
preservation:
regular $\al$ in $I$,
singular $\al$ in $K$, and
singular $\al$ in $I\smallsetminus K$.
However, for the first two of these, 
the
forcing iteration is trivial at stage $\al$, and the question
reduces to lifting embeddings 
$\pi:(H_{\bar\al^+},\in,\bar\sigma)\to(H_{\al^+},\in,\sigma)$
for $\P_\al$ names $\sigma$.  
We wish to show that it is dense in $\P_\al$ to force such a $\pi$ to
lift, so let $p$ be an arbitrary condition in $\P_\al$.
As in the proof of Theorem~\ref{nosubSRforc}, the support of $p$ is bounded
below $\ka$, and conditions $\bar p$ for embeddings 
$\pi:\,$\mbox{$(H_{\bar\al^+},\in,\{\ga\},\{\bar p\},\bar\sigma)\,$}$\to\,$
\mbox{$(H_{\al^+},\in,\{\ga\},\{p\},\sigma)$} with $\bar\ka$ minimal
are dense below $p\restr\ka$
as $\ga$ ranges over ordinals less than $\ka$.
Thus we may assume that $\pi$ is such an embedding with $\bar p\in G$.
The structure $H_{\al^+}$ correctly computes
$I\cap\al$ and $K\cap\al$, so $\pi(\P_{\bar\al})=\P_\al$.
Let $G_{[\bar\ka,\bar\al)}$ denote the generic over $V[G_{\bar\ka}]$
for $\P^{[\bar\ka,\bar\al)}$, the part of the iteration from stage
$\bar\ka$ up to but not including stage $\bar\al$.
Note that the non-trivial iterands in $\P^{[\ka,\al)}$ are all of the form 
$\bbT_\be$ for some singular $\be\in I\smallsetminus K$, that is,
singular $\be$ with $\cf(\be)\geq\ka$.
Since directed closure iterates 
(see for example \cite[Proposition~7.11]{Cum:IFE}), we have that
$\P^{[\ka,\al)}$ is $<\ka$ directed closed.
Hence, there is a condition in $\P^{[\ka,\al)}$
extending every condition in $\pi``G_{[\ka,\al)}$, including in particular
$p$, since $G$ was assumed to contain $\bar p$.
This condition is the desired master condition extending $p$.

For singular $\al\in I\smallsetminus K$, the argument is not too different.
Let $\ka$ be the least cardinal that is $\al^+$-subcompact.
In this case $\sigma$ will be a $\P_{\al+1}\cong\P_\al*\dot\bbT_\al$-name,
so $\bar\sigma$ will be $\P_{\bar\al}*\dot\bbT_{\bar\al}$-name.
As in the proof of Theorem~\ref{nosubSRforc}, we may take $\pi$ 
witnessing $\al^+$-subcompactness of $\ka$ such that 
no $\ka'$ is $\bar\al^+$-subcompact.
Thus, the forcing will be $\bbT_{\bar\al}*\dot\bbS_{\bar\al}$ 
at stage $\bar\al$,
and from $G$ we get a $\P_{\bar\al}*\dot\bbT_{\bar\al}$-generic,
which gives rise to master condition in $\P^{[\ka,\al+1)}$.
As usual this argument can be run below any condition $p\in\P_{\al^+}$,
so such master conditions are dense, and $\al^+$-subcompactness of 
$\bar\ka$ is preserved.
\end{proof}

In this case our forcing will be non-trivial on certain singular cardinals
between $\ka$ and $\la$ (as in Definition~\ref{omss}) for $\ka$
an $\omega$-superstrong cardinal.  However, it seems likely that a careful
homogeneity argument, using a homogeneity iteration result like those in
\cite{DoF:HIM1C}, will show that $\omega$-superstrong cardinals are again 
preserved under this forcing; we leave the details to the interested reader.

\bibliographystyle{asl}
\bibliography{SubSqSR}

\end{document}